\documentclass{endm}
\usepackage{endmmacro}

\usepackage[latin1]{inputenc}
\usepackage{multirow}

\usepackage{amssymb,amsmath}
\usepackage{epsfig}
\usepackage{color}
\usepackage{url}

\newtheorem{thma}{Theorem}
\newtheorem{lema}[thma]{Lemma}
\newtheorem{propo}[thma]{Proposition}

\newtheorem{question}{Question}

\input amssym.def

\newcommand{\B}{{\cal B}}
\newcommand{\bn}{b}

\begin{document}

\begin{verbatim}\end{verbatim}\vspace{2.5cm}

\begin{frontmatter}

\title{The b-continuity of graphs with large girth \thanksref{ALL}}

\author[Pargo,DM]{Ana Silva\thanksref{email}}
\author[Pargo,DC]{Claudia Linhares Sales\thanksref{email}}

  \thanks[email]{\emph{Email addresses:}
  \href{}{\texttt{\normalshape anasilva@mat.ufc.br} (Silva), \texttt{\normalshape linhares@lia.ufc.br} (Linhares Sales)}} 
 
\address[Pargo]{ParGO Research Group - Parallelism, Graphs and Optimization} 
\address[DM]{Departamento de Matem\'atica, Universidade Federal do Cear\'a, Brazil}
\address[DC]{Departamento de Computa\c{c}\~ao, Universidade Federal do Cear\'a, Brazil}

\thanks[ALL]{Partially supported by CNPq/Brazi.}

\begin{abstract}
A b-coloring of the vertices of a graph is a proper coloring where each color class contains a vertex which is adjacent to each other color class. The b-chromatic number of $G$ is the maximum integer $\bn(G)$ for which $G$ has a b-coloring with $\bn(G)$ colors. A graph $G$ is b-continuous if $G$ has a b-coloring with $k$ colors, for every integer $k$ in the interval $[\chi(G),\bn(G)]$. It is known that not all graphs are b-continuous. In this article, we show that if $G$ has girth at least 10, then $G$ is b-continuous.
\end{abstract}

\end{frontmatter}

\section{Introduction}

Let $G$ be a simple graph\footnote{The graph terminology used in this paper follows \cite{BM08}.}. A \emph{proper coloring of $G$} (from now on called coloring) is a function $\psi:V(G)\rightarrow \mathbb{N}$ such that $\psi(u)\neq \psi(v)$ whenever $uv\in E(G)$. We say that $u\in V(G)$ is a \emph{b-vertex in $\psi$ (of color $\psi(u)$)} if for every color $c\neq \psi(u)$, there exists $v$ colored with $c$ that is adjacent to $u$. 
Observe that if $\psi$ has a color class $c$ that has no b-vertices, then we can separatedly change the color of each vertex in $c$ to obtain a proper coloring with fewer colors.  But since the coloring problem is NP-complete, this clearly cannot always be applied until one reaches a coloring of $G$ with $\chi(G)$ colors. An interesting case of study therefore is to investigate the worst possible scenario. More formally, Irving and Manlove ~\cite{Irving.Manlove.99} defined a \emph{b-coloring} as a coloring of $G$ that has at least one b-vertex in each of its color classes, and the \emph{b-chromatic number of $G$}, denoted by $\bn(G)$, as the maximum number of colors used by a b-coloring of $G$. They showed that the problem is NP-complete in general, but polynomial-time solvable for trees. After this, the problem has been shown to be NP-complete also for bipartite graphs \cite{KRATOCHVIL.etal.02},  chordal graphs \cite{HLS.11}, and line graphs~\cite{CLMSSS.15}.

Irving and Manlove also observed that the cube has b-colorings with 2 and 4 colors, but not with 3 colors, and in~\cite{KRATOCHVIL.etal.02} the authors show that $K_{n,n}$ minus a perfect matching only admits b-colorings with 2 and $n$ colors, for every positive integer $n$. This inspired the following definitions: the \emph{b-spectrum} of a graph $G$ is the set $S_b(G)$ containing every $k$ for which $G$ has a b-coloring with $k$ colors; and $G$ is called \emph{b-continuous} if $S_b(G)$ contains every integer in the interval $[\chi(G),\bn(G)]$. In~\cite{Barth.Cohen.Faik}, the authors prove that in fact for every finite subset $S\subset \mathbb{N}-\{1\}$, there exists a graph $G$ such that $S_b(G) = S$. They also prove that deciding whether a given graph $G$ is b-continuous is an NP-complete problem, even if b-colorings with $\chi(G)$ and $b(G)$ colors are given. Concerning positive results, we mention that the following graph classes are b-continuous: chordal graphs~\cite{F.04,KKV.04}; Kneser graphs $K(n,2)$ for $n\ge 17$~\cite{JO.09}; $P_4$-sparse graphs~\cite{Bonomo.etal.09} and $P_4$-tidy graphs~\cite{BBK.11}; and regular graphs with girth at least 6 and with no cycles of length 7~\cite{BK.12}. 

The \emph{girth of $G$} is the minimum size of a cycle in $G$. Observe that a necessary condition for $G$ to have a b-coloring with $k$ colors is that $G$ also has at least $k$ vertices with degree at least $k-1$, namely the b-vertices in such a b-coloring. Therefore, if $m(G)$ is defined as the maximum value for which $G$ satisfies this condition, i.e., $m(G)$ is the largest integer $k$ such that $G$ has at least $k$ vertices of degree at least~$k-1$, then naturally we get $\bn(G)\le m(G)$. One can verify that the value $m(G)$ can be easily computed given the degree sequence of $G$ in non-increasing order.
As mentioned before, Irving and Manlove proved in their seminal paper that finding the b-chromatic number of a tree can be done in polynomial time. In fact, they prove that if $G$ is a tree, then $\bn(G)$ is at least $m(G)-1$, and one can decide whether $\bn(G)$ equals $m(G)$ in polynomial time. Later, it was noted that in fact this property holds for graphs with large girth, and the most recent  result regarding this aspect says that $\bn(G)\ge m(G)-1$ holds whenever $G$ has girth at least~7~\cite{CLS.15}. We believe this to be already a good reason for the investigation of the b-continuity of graphs with large girth, but in~\cite{HLS.11} and~\cite{EK.09} there are further questions concerning the b-chromatic number of graphs with large girth. 
In~\cite{HLS.11}, the authors conjecture that if $G$ belongs to a certain subclass of the bipartite graphs with girth at least~6, then $\bn(G)\ge m(G)-1$. Although the target subclass is a fairly simple one, up to now the only existing result on their conjecture is given in~\cite{LC.13}, where the authors prove that their conjecture is a consequece of the famous Erd\H{o}s-Faber-Lov\'asz Conjecture~\cite{E.81}. 
In~\cite{EK.09}, El Sahili and Kouider ask whether $\bn(G)=m(G)$ whenever $G$ is a $d$-regular graph with girth at least~5. Blidia, Maffray and Zemir~\cite{BMZ.09} answered their question in the negative by showing that the Petersen graph does not have this property; then they conjecture that this is the only exception. Since then, a number of partial results have been given; for instance~\cite{CJ.11,BMZ.09,EKM.15,S.12}. 
All these results on graphs with large girth indicate that, unlike the classic coloring problem, having large girth somehow helps in finding b-colorings of $G$. This is why we ask the following question, similar to the one posed in~\cite{CLS.15}.

\begin{question}
What is the minimum $\hat{g}$ such that $G$ is b-continuous whenever $G$ has girth at least $\hat{g}$?
\end{question}

In this article, we prove that $\hat{g}\le 10$. This and the fact that complete bipartite graphs without a perfect matching are not b-continuous~\cite{KRATOCHVIL.etal.02} give us that $5\le \hat{g}\le 10$. Up to our knowledge, no non-b-continuous graph with girth at least~5 is known.

\begin{thma}\label{main}
If $G$ has girth at least 10, then $G$ is b-continuous.
\end{thma}

To prove the above theorem, we first analyse a special case, then we prove that if $G$ has a b-coloring with $k$ colors and $k\ge \chi(G)+2$, then either we can construct a b-coloring with $k-1$ colors by applying certain operations, or $G$ falls into the special case. We emphasize that these operations actually hold for girth 6, i.e., the stronger hypothesis of the theorem is only necessary to treat the special case.

We mention that in~\cite{CLS.15}, the authors also give a polynomial-time algorithm to compute the b-chromatic number of graphs with girth at least~7. At the same time, computing the chromatic number of a graph with girth at least $g$ is NP-complete, for every fixed $g\ge 3$, even if $G$ is a line-graph~\cite{LK.07}. Therefore, any proof of a result as Theorem~\ref{main} must have a non-constructive part, as otherwise one can start with an optimal b-colouring of graph $G$, and decrease the number of used colors until one obtains an optimal coloring of $G$. The non-constructive part of our proof lies in the construction of a b-coloring for the special case.

Before we proceed, we need further definitions. Given a coloring $\psi$ of $G$ with $k$ colors, a value $i$ in $\{1,\cdots,k\}$ is called a \emph{color}, the set  $\psi^{-1}(i)$ is called the \emph{color class $i$}, and given a subset $X\subseteq V(G)$, we denote by $\psi(X)$ the set $\{\psi(v)\mid v\in X\}$. We say that a color is \emph{realized in $\psi$} if $\psi^{-1}(i)$ contains a b-vertex $u$; we also say that \emph{$u$ realizes color $i$}.  If $\psi$ is a coloring of $G$ with $k$ colors and $i\in\{1,\cdots,k\}$ is such that $i$ is not realized in $\psi$, then a coloring of $G$ with $k-1$ colors can be obtained by changing the color of each $u\in \psi^{-1}(i)$ to some $j\in \{1,\cdots,k\}\setminus \psi(N[u])$; we say that the new coloring is obtained from $\psi$ by \emph{cleaning color $i$}. 
The \emph{distance} between vertices $u$ and $v$ in $G$ is given by the minimum number of edges in a path between $u$ and $v$ in $G$; if $u$ and $v$ are in distinct components of $G$, we say that their distance is $\infty$. Given $u\in V(G)$, we denote by $N_i(u)$ the set of vertices at distance exactly $i$ from $u$, and by $N_{\le i}(u)$ the set $\bigcup_{j=1}^iN_i(u)$. Finally, we say that $u\in V(G)$ is \emph{$k$-dense} if $d(u)\ge k-1$, and we denote the set of $k$-dense vertices of $G$ by $D_k(G)$.

\section{Special case}

Let $G$ be any graph. We say that $u\in V(G)$ is a \emph{$k$-iris} in $G$ if there exists $S\subseteq N(u)\cap D_k(G)$ with cardinality $k-1$; and we say that $u$ is a \emph{dilated $k$-iris} if there exists a subset $S\subseteq N_{\le 2}(u)\cap D_k(G)$  with $k$ vertices such that $N(v)\cap N(w)=\emptyset$, for every $v,w\in S\cap N_2(u)$, $v\neq w$. Observe that if $u$ is a $k$-iris or a dilated $k$-iris, then $u\in D_k(G)$.

Let $u$ be a (dilated) $k$-iris and $G'$ be the induced subgraph $G[N_{\le 3}(u)]$. Roughly, the idea of the proof is to construct a b-coloring of $G'$ with $k$ colors, then use color 1 to color $N(G')$, and color the remaining vertices with $\{2,3,\cdots,k\}$, which is possible as long as $k\ge \chi(G)+1$. In~\cite{BK.12}, the authors prove that if $G$ is a regular graph with girth at least~6 and containing no cycles of length~7, then $G$ is b-continuous. For this, they show that, for any integer $k$ in the interval $[\chi(G),\bn(G)]$, a b-coloring with $k$ colors can be obtained. Their proof actually holds whenever $G$ has a $k$-iris, as can be seen in the next lemma. We present their proof for completeness sake, since our notation differs from theirs.

\begin{lema}[\cite{BK.12}]\label{lem:iris}
Let $G$ be a graph with girth at least 6 and with no cycles of length~7. If $G$ has a $k$-iris and $k\ge \chi(G)$, then $G$ has a b-coloring with $k$ colors.
\end{lema}
\begin{proof}
If $k=\chi(G)$, then every $k$-coloring of $G$ is also a b-coloring of $G$; hence, suppose that $k\ge \chi(G)+1$. Let $u\in V(G)$ be a $k$-iris, and let $S= \{v_2,\cdots,v_k\}\subseteq N(u)\cap D_k(G)$. 

Start by coloring $u$ with 1 and $v_i$ with $i$, for each $i\in \{2,\cdots,k\}$. Then, for each $i\in\{2,\cdots,k\}$, color $k-2$ uncolored vertices in $N(v_i)$ with the colors missing in $N(v_i)$. Note that this can be done because $G[N_{\le 2}(u)]$ is a tree. Let $\psi$ be the obtained partial coloring, $T$ be the set of colored vertices, and $U$ be the set of uncolored vertices in $N(T-\{u\})$. Note that, because $G$ has girth at least 6 and contains no cycles of length 7, we get that $U\cup\{u\}$ is a stable set. Also, note that $u$ is the only vertex colored with $1$ in $T$, which means that $1\notin \psi(N(v))$ for every $v\in U$. Hence, we can give color 1 to every vertex in $U$. Finally, note that $U\cup \{u\}$ separates $T-\{u\}$ from the rest of the graph, which implies that we can use colors $\{2,\cdots,k\}$ to color the remaining vertices.
\end{proof}

In order to color the dilated $k$-iris, we use a similar approach. However, because some of the vertices of $S$, which are the vertices that we want to turn into b-vertices, are at distance~2 from $u$, we need a stronger restriction on the girth of $G$ in order to have the ``barrier'' of vertices colored with color~1.

\begin{lemma}\label{lem:dilatediris}
Let $G$ be a graph with girth at least~10. If $G$ has a dilated $k$-iris and $k\ge \chi(G)$, then $G$ has a b-coloring with $k$ colors.
\end{lemma}
\begin{proof}
Again, we can suppose that $k\ge \chi(G)+1$. Let $u\in V(G)$ be a dilated $k$-iris, and let $S = \{v_1,\cdots,v_k\}\subseteq N_{\le 2}(u)\cap D_k(G)$ satisfying the conditions of the definition. We can suppose that $N_{= 2}(u)\cap S\neq \emptyset$ as otherwise $u$ is also a $k$-iris and the lemma follows by Lemma~\ref{lem:iris}. So suppose, without loss of generality, that $\{v_1,\cdots,v_p\}\subseteq N(u)$ and $\{v_{p+1},\cdots,v_k\}\subseteq N_2(u)$. Also, for each $i\in \{p+1,\cdots,k\}$, let $w_i\in N(v_i)\cap N(u)$, and denote by $W$ the set $\{w_{p+1},\cdots,w_k\}$. 

Start by coloring $u$ with $p+1$, $v_i$ with $i$, for each $i\in \{1,\cdots,k\}$, and $w_i$ with $1$, for each $i\in\{p+1,\cdots,k\}$. Let $v_i,v_j\in S$ and consider $x\in N(v_i)\setminus (W\cup \{u\})$ and $y\in N(v_j)\setminus (W\cup \{u\})$. Because $G$ has girth at least 10 and $x,y\in N_{\le 3}(u)$, we get: $v_i$ is the only neighbor of $x$ in $S\cup W\cup \{u\}$ (the same for $y$ and $v_j$); and $xy\notin E(G)$. Therefore, for each $i\in\{1,\cdots,k\}$, we can easily color $k-2$ neighbors of $v_i$ in such a way as to turn $v_i$ into a b-vertex. 

Now, denote by $C$ be the set of colored vertices, by $C_1$ the set of vertices colored with $1$, and by $U$ the set of uncolored vertices in $N(C \setminus C_1)$. Note that $C_1\cup U\subseteq N_{\le 4}(u)$; therefore, if $x$ and $y$ are distinct vertices in $C_1\cup U$, then their distance to $u$ is at most 4 and, since $G$ has girth at least 10, we get $xy\notin E(G)$. This means that $C_1\cup U$ is a stable set, in which case we can color each vertex of $U$ with color~$1$. Finally, by the choice of $U$, we know that $U\cup C_1$ separates $C\setminus C_1$ from the rest of the graph, and since $k\ge \chi(G)+1$, we can color the remaining vertices with $\{2,\cdots,k\}$
\end{proof}

In the next section, we start with a b-coloring of $G$ with $k+1$ colors, $k\ge\chi(G)$, and try to obtain a b-coloring of $G$ with $k$ colors. In the cases we are not able to do that, we are able to prove that $G$ has either a $k$-iris or a dilated $k$-iris; the theorem thus follows.

\section{Decreasing the number of realized colors}\label{sec:decreasing}

Let $\psi$ be any coloring of $G$ with $k$ colors, and let $\B(\psi)$ denote the set of all b-vertices in $\psi$. If $\psi$ is a b-coloring and $\psi$ minimizes $\lvert \B(\psi)\rvert$, we say that $\psi$ is \emph{minimal}.  If under certain conditions we can decrease the number of b-vertices used in a minimal b-coloring $\psi$, and we can ensure that only b-vertices of a certain color are lost, then we know that the obtained coloring realizes exactly $k-1$ colors. This means that a b-coloring with $k-1$ colors can be obtained by cleaning the color class that is not realized in $\psi$. Next, we describe a number of situations where such a coloring can be obtained. 

Consider $\psi$ to be a minimal b-coloring with $k$ colors. In what follows, we refer only to $\psi$, so we ommit it from the definitions. For each color $i$ denote by $B_i$ the set of b-vertices in color class $i$, i.e., $B_i = \psi^{-1}(i)\cap \B(\psi)$. Also, for each $x\in V(G)\setminus \B(\psi)$, let $U(x)$ contain each $w\in \B(\psi)$ such that $x$ is the only neighbor of $w$ colored with $\psi(x)$. In other words, letting $c=\psi(x)$:
\[U(x) = \{w\in \B(\psi)\mid \psi^{-1}(c)\cap N(w) = \{x\}\}\]

In the following arguments, we usually change the color of some $x\in V(G)\setminus \B(\psi)$ and want to conclude that $\lvert \B(\psi)\rvert$ decreases, while preserving at least $k-1$ realized colors. For this, we want to ensure that some $B_i$ disappears, but also that no new b-vertex is created. The following straightforward proposition is useful.

\begin{propo}\label{prop:nonewBvertex}
Let $\psi$ be a coloring of $G$ with $k$ colors and $x\in V(G)\setminus \B(\psi)$. If $\psi'$ is a coloring of $G$ with $k$ colors such that $\lvert \psi'(N[x])\rvert = \lvert \psi(N[x])\rvert$, then $x\notin \B(\psi')$.
\end{propo}

This proposition ensures that if vertex $x$ has its color changed, but nothing else changes in its neighborhood, then $x$ does not turn into a b-vertex. However, this can still happen to a neighbor $y$ of $x$ as the number of colors appearing in $N(y)$ may increase. We say that $x\in V(G)\setminus \B(\psi)$ is \emph{mutable} if there exists $i\in \{1,\cdots,k\}\setminus\psi(N[x])$ such that $\{1,\cdots,k\}\setminus\psi(N[w]\setminus\{x\})\neq \{i\}$, for every $w\in N(x)\setminus \B(\psi)$; also, color~$i$ is said to be \emph{safe for $x$}. This means that changing the color of $x$ to $i$ does not create new b-vertices.  

\begin{propo}\label{prop:unique}
Let $\psi$ be a b-coloring of $G$ with $k$ colors, $k\ge \chi(G)+1$. If $x\in V(G)\setminus \B(\psi)$ is mutable and $\lvert \psi(U(x))\rvert = 1$, then there exists a b-coloring of $G$ with $k-1$ colors.
\end{propo}
\begin{proof}
Let $\psi(U(x)) = \{d\}$. Just change the color of $x$ to any safe color $c\notin \psi(N[x])$, and let $\psi'$ be the obtained coloring. Proposition \ref{prop:nonewBvertex} clearly applies. Also, by the definition of $U(x)$, we know that every $w\in U(x)$ is not a b-vertex in $\psi'$ since color $\psi(x)$ does not appear in $\psi'(N(w))$. Finally, by the definition of safe color we know that $\B(\psi')\subset \B(\psi)$. Therefore, only b-vertices of color $d$ are lost, no new b-vertex is created, and because $\psi$ is minimal, we know that there is no remaining b-vertices with color $d$. Thus, the desired b-coloring can be obtained by cleaning color $d$.
\end{proof}

Now, for each $u\in \B(\psi)$ and each color $i\in \{1,\cdots,k\}\setminus\psi(u)$, let $B_i(u)=N(u)\cap B_i$, and $R_i(u)$ be the remaining neighbors of $u$ colored with $i$, i.e., $R_i(u) = (N(u)\setminus B_i)\cap \psi^{-1}(i)$. We say that color $i$ is \emph{weak in $N(u)$} if, for every $x\in R_i(u)$, $x$ is mutable and the following occurs:
\begin{enumerate}
  \item[(*)] For every $w\in U(x)\setminus \{u\}$, there exists $w'\in \B(\psi)$ that realizes color $\psi(w)$ and such that $w'\notin N(R_i(u))$.
\end{enumerate}

Condition (*) means that, if the color of every $x\in R_i(u)$ gets changed, then not every b-vertex of color $\psi(w)$ is lost. 

\begin{propo}\label{prop:weak}
Let $\psi$ be a b-coloring of $G$ with $k$ colors, $k\ge \chi(G)+1$, and $u\in \B(\psi)$. If $i\in\{1,\cdots,k\}\setminus\{\psi(u)\}$ is weak in $N(u)$ and $B_i(u)=\emptyset$, then there exists a b-coloring of $G$ with $k-1$ colors.
\end{propo}
\begin{proof}
We change the color of each $x\in R_i(u)$ to any safe color $c\in\{1,\cdots,k\}\setminus \psi(N[x])$. Because $R_i(u)$ is a stable set, Proposition \ref{prop:nonewBvertex} applies to each $x\in R_i(u)$. Also, if $j\notin \psi(U(x)\setminus\{u\})$, we know that $j$ does not lose any b-vertex; otherwise, by Condition (*) above, we know that there still exists some vertex that realizes color $j$. Now, $u$ has no more neighbors of color $i$, which by the minimality of $\psi$ implies that no more b-vertices of color $\psi(u)$ exists. Thus, the desired b-coloring can be obtained by cleaning color $\psi(u)$.
\end{proof}

The next lemma finishes the proof. 

\begin{lema}
Let $G$ be a graph with girth at least~5, and $\psi$ be a b-coloring of $G$ with $k$ colors, $k\ge \chi(G)+1$. Then either there exists a b-coloring of $G$ with $k-1$ colors, or $G$ has a $(k-1)$-iris, or a dilated $(k-1)$-iris.
\end{lema}
\begin{proof}
Let $u$ be any b-vertex in $G$ and suppose, without loss of generality, that $\psi$ is minimal. Let $S=\{i\mid B_i(u)=\emptyset\}$. First, note that if $\lvert S\rvert\le1$, then $u$ has at least $k-2$ neighbors which have degree at least $k-1$, i.e., $u$ is a $(k-1)$-iris and we are done. So suppose otherwise, and let $x_i\in R_i(u)$, for some $i\in S$. If $i$ is a weak color in $N(u)$, the lemma follows by Proposition \ref{prop:weak}. Otherwise, either $x_i$ is unmutable, or (*) does not hold for $x_i$. In both cases, one can see that there exists $v_i \in N(x_i)\cap D_k(G)$. Since $G$ has girth at least~5, we know that $v_i\notin N(u)$, for every $i\in S$, and that $v_i\neq v_j$, for every $i,j\in S$, $i\neq j$. Therefore, we get that $u$ is a dilated $(k-1)$-iris.
\end{proof}

\section{Conclusion}

We have proved that graphs with girth at least~10 are b-continuous, giving yet more evidence that being locally acyclic helps at obtaining b-colorings of the graph, contrary to the classical coloring problem. We mention that the girth~10 condition is needed in only one small part of the proof, namely to color a dilated $k$-iris. However, one can observe that the arguments applied in Section~\ref{sec:decreasing} concern the vicinity of only one vertex. This means that when Propositions~\ref{prop:unique} and~\ref{prop:weak} cannot be applied, then the graph $G$ is a lot more structured, namely that every mutable vertex $x$ has $\lvert \psi(U(x))\rvert\neq 1$, and every b-vertex $u$ is the center of either a $(k-1)$-iris or a dilated $(k-1)$-iris. In fact, we believe that even the ``quasi-b-vertices'' (vertices in the neighborhood of unmutable vertices) also generate irises. This is why we believe that our result can be generalized for every graph with girth at least~6 and with no cycles of length~7, which would generalize the result in~\cite{BK.12}. Additionally, given the interest in regular graphs and bipartite graphs with large girth, we pose the following questions.

\begin{question}
Are regular graphs with girth at least~5 b-continuous?
\end{question}

\begin{question}
Are bipartite graphs with girth at least~6 b-continuous?
\end{question}


\end{document}